\documentclass{elsarticle}
\usepackage{amsmath}
\usepackage{amsthm}
\usepackage{mathrsfs}
\usepackage{amsfonts}
\usepackage{thm-restate}

\begin{document}
\title{Cleavability over ordinals}
\author{Shari Levine}
\address{Mathematical Institute, University of Oxford, Oxford, OX1 3LB}
\begin{abstract} In this paper we show that if $X$ is an infinite compactum cleavable over an ordinal, then $X$ must be homeomorphic to an ordinal. $X$ must also therefore be a LOTS. This answers two fundamental questions in the area of cleavability. We also leave it as an open question whether cleavability of an infinite compactum $X$ over an ordinal $\lambda$ implies $X$ is embeddable into $\lambda$. \end{abstract}

\maketitle
\newtheorem{thm}{Theorem}[section]
\newtheorem{lemma}[thm]{Lemma}
\newtheorem{definition}[thm]{Definition}
\newtheorem{obs}[thm]{Observation}
\newtheorem{cor}[thm]{Corollary}
\newtheorem{question}{Question}
\newtheorem{claim}[thm]{Claim}
\newtheorem{ex}[thm]{Example}
\newtheorem{prop}[thm]{Proposition}
\newtheorem{cex}[thm]{Counter-Example}

\newcommand{\reals}{\mathbb{R}}
\newcommand{\cf}{\textnormal{cf}}
\newcommand{\pcf}{\textnormal{pcf}}
\newcommand{\CB}{\textnormal{CB}}

\section*{Introduction}

A space $X$ is said to be \textit{cleavable} (or splittable) over a space $Y$ \textit{along} $A \subseteq X$ if there exists a continuous $f: X \rightarrow Y$ such that $f(A) \cap f(X \setminus A) = \emptyset$. A space $X$ is \textit{cleavable over} $Y$ if it is cleavable over $Y$ along all $A \subseteq X$. The topic was introduced by A.~V.~Arhangel'ski\u\i\ and D.~B.~Shakh\-matov in \cite{splitting}, but it was in \cite{Arhangelskii2} that two of the main questions were stated:

\begin{question}\label{qone} Let $X$ be an infinite compactum cleavable over a linearly ordered topological space (LOTS) $Y$. Is $X$ homeomorphic to a subspace of $Y$? \end{question}

\begin{question}\label{qtwo} Let $X$ be an infinite compactum cleavable over a LOTS. Is $X$ a LOTS? \end{question}

Results related to these questions can be found in, but are not limited to, the following papers: \cite{Arhangelskii2}, \cite{Arhangelskii}, \cite{buzyakova}.

In this paper, we show that if $X$ is an infinite compactum cleavable over an ordinal, then $X$ must be homeomorphic to an ordinal. These results supplement those in [Levine, \textit{Cleavability over scattered first-countable LOTS}, unpublished], which concerns cleavability over ordinals less than or equal to $\omega_1$. 

Providing positive answers to Questions~\ref{qone} and \ref{qtwo} allows for the comprehension of a difficult space $X$ by associating it with a more well-known and well-understood space: in this case an uncountable ordinal. 

Further, several papers have been devoted to describing the necessary and sufficient conditions for the linear orderability of a space $X$. (See \cite{os} and \cite{lb} for examples.) The results of this paper provide an alternative characterization by which we may show when an infinite compactum is linearly orderable. 

Additionally, another popular area of research is the characterization of those spaces that are homeomorphic to an ordinal (see \cite{contributions}, \cite{D}, and \cite{moran} for examples.) The results of this paper adds another characterization of those spaces homeomorphic to an ordinal. 

The novelty of these characterizations is that, whereas the results of \cite{os}, \cite{lb}, \cite{contributions}, \cite{D}, and \cite{moran} rely on topological properties of a space $X$, our new characterizations rely on finding an appropriate ordinal $\lambda$, and an appropriate subset of $\mathscr{C}(X,\lambda)$. This shifts the focus from a topological exercise, to a functional one.

This paper is written in four sections. In the first section, we provide introductory definitions, observations and lemmas. The most important of these is Theorem~\ref{amain}; in this theorem, we show that any compact $X$ cleavable over an ordinal $\lambda$ such that $X$ ``hereditarily has a spine'' must be homeomorphic to an ordinal. (A definition for this property is provided in Definition~\ref{hspine}.) In the second and third sections, we show that every compact $X$ cleavable over an ordinal must hereditarily have a spine. In the fourth section, we provide an answer to Questions~\ref{qone} and \ref{qtwo}. The second and third sections of this paper are heavily technical. The main result of this paper, that $X$ is homeomorphic to an ordinal $Y$, is stated and proven in Theorem~\ref{othm}.

\section{Introductory Proofs}

In this section we provide introductory definitions, observations, and lemmas. The most important lemma of this section is Theorem~\ref{amain}, in which we show that every infinite compactum that ``hereditarily has a spine'' (see Definition~\ref{hspine}), and is cleavable over an ordinal, must be homeomorphic to an ordinal. This provides the foundation for the rest of the paper, in which we prove that every infinite compactum cleavable over an ordinal must hereditarily have a spine.  

We begin by stating several well-known definitions, observations, and lemmas. The first theorem is from \cite{Arhangelskii}. 

\begin{thm} If $X$ is a compactum cleavable over a $T_2$ space, then $X$ is $T_2$. \end{thm}

\begin{lemma} \label{scattered} If $X$ is a compactum cleavable over a scattered space $Y$, then $X$ must be scattered. \end{lemma}

\begin{proof} Assume for a contradiction that $X$ contains a dense-in-itself subset $D$, and consider $\overline{D}$. We know $\overline{D}$ is compact $T_2$, and perfect, therefore by \cite{resolvable} it is resolvable. Let $\overline{D} = A \cup B$, where $\overline{A} = \overline{B} = \overline{D}$, and $A \cap B = \emptyset$. No function can cleave apart $A$ and $B$, thus $X$ cannot contain a dense-in-itself subset.\end{proof}

\begin{definition} For ordinal numbers $\alpha$, the $\alpha$-th \textbf{derived set} of a topological space $X$ is defined by transfinite induction as follows:
\begin{itemize}
\item $X^0 = X$
\item $X^{\alpha+1} = (X^{\alpha})'$
\item $\displaystyle X^{\lambda} = \bigcap_{\alpha < \lambda} X^{\alpha}$ for limit ordinals $\lambda$.\end{itemize}

The smallest ordinal $\alpha$ such that $X^{\alpha+1} = X^{\alpha}$ is called the \textbf{Cantor-Bendixson rank} of $X$, written as $\CB(X)$. 
Lastly, let \textbf{$I_{\beta}(X)$} $= X^{\beta} \setminus X^{\beta+1}$. \end{definition}

The following observations are well known, but may also be found in \cite{sierpinski}.

\begin{obs} For a scattered space $X$, the Cantor-Bendixson rank is the least ordinal $\mu$ such that $X^{\mu}$ is empty. \end{obs}

\begin{obs} \label{o2} If $X$ is a compact scattered topological space, then the Cantor-Bendixson rank of $X$ must be a successor ordinal. \end{obs}

It follows from these observations that if $X$ is compact and scattered, and $\CB(X) = \beta+1$, then $X^{\beta}$ is the last non-empty derived set of $X$. By compactness, $X^{\beta} < \omega$; from this, we have the following definitions:

\begin{definition} Let $X$ be a compact scattered space. If $\CB(X) = \beta+1$, we say $X$ is \textbf{simple} if $|X^{\beta}| = 1$. We say $X$ is \textbf{simple with $\hat{x}$} if $\hat{x}$ is the only element of $X^{\beta}$. \end{definition}

\begin{definition} Let $X$ be a scattered topological space. For $x \in X$, we use $\CB^*(x)$ to be the greatest ordinal $\beta$ such that $x \in X^{\beta}$. \end{definition}

\begin{definition} Let $X$ be a compact scattered space. If $X^{\beta}$ is finite and contains exactly $n$-many points, the pair $(\beta,n)$ is called the \textbf{characteristic} of $X$. \end{definition}

The following Lemma is also easily provable, but may be found in \cite{D}:

\begin{lemma} \label{ord} Any closed subset of an ordinal is homeomorphic to an ordinal. \end{lemma}

We now state an important definition.

\begin{definition} \label{spine} Let $X$ and $Y$ be such that $X$ is infinite, compact, and simple with $\hat{x}$, $Y$ is an infinite ordinal, and $X$ is cleavable over $Y$. We say $X$ has a \textbf{spine} $(S,k)$, where $S$ is a subset of $X \setminus \left\{\hat{x}\right\}$, and $k: X \rightarrow Y$ is a continuous function, if the following properties are satisfied: \begin{enumerate}

\item $k$ cleaves along $S$
\item $k|_S$ is an embedding into $Y$
\item $k(S)$ is club in $k(\hat{x})$. \end{enumerate} \end{definition}

For example, if $k$ were injective on $X$, then $(X \setminus \left\{\hat{x}\right\},k)$ would satisfy this definition. It may aid the reader to think of $(S,k)$ as creating what looks like a spine in $k(X)$.

\begin{definition} \label{hspine} Let $X$ and $Y$ be infinite spaces such that $X$ is compact and simple with $\hat{x}$, $Y$ is an infinite ordinal, and $X$ is cleavable over $Y$. We say $X$ \textbf{hereditarily has a spine} if every closed, infinite, simple $A \subseteq X$ has a spine. \end{definition}

\begin{definition} Let $X$ and $Y$ be such that $X$ is infinite, compact, and simple with $\hat{x}$, $Y$ is an infinite ordinal, and $X$ is cleavable over $Y$. We say $(R,j)$ is a \textbf{semi-spine} of $X$ if $(R,j)$ satisfies properties $(2)$ and $(3)$ of Definition~\ref{spine}. \end{definition}. 

The following definition and theorem may be found in \cite{D}.

\begin{definition} We shall say that a point $x \in X$ satisfies $(D)$ in $X$ if $x$ has a neighborhood base consisting of a decreasing, possibly transfinite, sequence $\left\{U_{\alpha} \right\}_{\alpha < \tau}$ of clopen sets with the additional property that $(\bigcap_{\alpha < \beta} U_{\alpha}) \setminus U_{\beta}$ contains at most one point for each limit ordinal $\beta$ with $\beta < \tau$. \end{definition}

\begin{thm} \label{D} Let $X$ be a compact scattered space with characteristic $(\lambda,n)$. If $X$ has property $(D)$, it is homeomorphic to $(\omega^{\lambda} \cdot n) +1$. \end{thm}

We are now ready to state the main result of this section, in which we describe how we will show any infinite compactum cleavable over an ordinal must be homeomorphic to an ordinal.  

\begin{thm} \label{amain} Let $X$ be a compact space cleavable over an ordinal. If $X$ hereditarily has a spine, then $X$ is homeomorphic to an ordinal. \end{thm}

\begin{proof} By Lemma~\ref{scattered}, we know $X$ must be scattered. As $X$ is compact, we may assume without loss of generality that $X$ is simple with $\hat{x}$. We will complete this proof using transfinite induction on $\CB(X)$. We must only consider the successor case, as by Observation~\ref{o2}, $\CB(X)$ must be a successor ordinal.

Let $\CB(X) = \alpha+1$, and assume we have shown that if $X$ hereditarily has a spine, and $\CB(X) < \alpha+1$, then $X$ is homeomorphic to an ordinal. Let $(S,k)$ be a spine for $X$. We may assume without loss of generality that $k(\hat{x})$ is the greatest element of $k(X)$. (Otherwise, we may modify $k$ such that this is true). As a consequence of being a spine for $X$, $S \cup \left\{\hat{x}\right\}$ must be closed in $X$; since $k$ is continuous, we thus know from Lemma~\ref{ord} that $k(S \cup \left\{\hat{x}\right\})$ is homeomorphic to an ordinal. Therefore, if $g$ is such a homeomorphism, and $\lambda+1$ is the ordinal to which it is homeomorphic, then enumerate the elements of $k(S)$ as $y_{\beta}$, where $g(y_{\beta}) = \beta \in \lambda$. Without loss of generality, assume $0 = y_0$. 

For each $\alpha \in I_0(\lambda)$, consider the interval $(y_{\alpha},y_{\alpha+1}]$. In the case when $\alpha = 0$, we consider the interval $[y_0,y_1]$. Note that these intervals may only contain a single element. Each interval is clopen, therefore $k^{-1}((y_{\alpha},y_{\alpha+1}])$ is clopen. Since this set is compact, scattered, by assumption has Cantor-Bendixson rank less than $X$, and hereditarily has a spine, this subset is homeomorphic to an ordinal $\lambda_{\alpha}$ by the inductive hypothesis. Let $h_{\alpha}$ be a homeomorphism from $k^{-1}((y_{\alpha},y_{\alpha+1}])$ to $\lambda_{\alpha}$. Replace $(y_{\alpha},y_{\alpha+1}]$ with a copy of $\lambda_{\alpha}$. Repeat this process for every clopen interval considered, and call the resulting space $Y$. Let the order on $Y$ preserve the order on each $\lambda_{\alpha}$, and the original order on $k(X)$. 

Let $Y$ be a LOTS. We will first show there exists a homeomorphism from $X$ to $Y$. We will then use Theorem~\ref{D} to show $Y$ is homeomorphic to an ordinal. Note that the least element of $Y$ is the least element of $\lambda_0$; call this element $0$. Also note that every element of $X$ is either an element of an aforementioned clopen interval, or $S'$, where $(S,k)$ is the spine of $X$. 

Let $\hat{h}_{\alpha}: k^{-1}((y_{\alpha},y_{\alpha+1}]) \rightarrow \lambda_{\alpha} \subset Y$ be identical to $h_{\alpha}$, and let $\hat{k}: S' \rightarrow Y \setminus \bigcup_{\alpha \in I_0(\lambda)} \lambda_{\alpha}$ be identical to $k$. Let $e$ be defined as: $$e(x) = \begin{cases}
\hat{h}_{\alpha}(x) & x \in k^{-1}((y_{\alpha},y_{\alpha+1}])\\
\hat{k}(x) & x \in S' \end{cases}$$

Injectivity of $e$ follows from the injectivity of each $\hat{h}_{\alpha}$, and $\hat{k}$; $e$ is also obviously well-defined. To show $e$ is continuous, as $Y$ is a LOTS, it is sufficient to show that $e^{-1}([0,a))$ and $e^{-1}((b,e(\hat{x})]))$ are open for every $a,b \in Y$. 

Let $[0,a) \subset Y$. Notice from the way we have created $Y$ that every point of $Y$ is either contained within some $\lambda_{\alpha}$, or is equal to $y_{\delta}, \delta \in {\lambda+1}'$. Therefore, there are two cases for $e^{-1}([0,a))$ to consider: \begin{enumerate}

\item If $a \in \lambda_{\alpha+1}, \alpha+1 \in I_0(\lambda)$, then $e^{-1}([0,a))$ is equal to $$k^{-1}([0,y_{\alpha})) \cup \left\{x \in X: h_{\alpha+1}(x) < a \right\}$$. The left set of the union is open by the continuity of $k$, and the right is open by the continuity of $h_{\alpha+1}$. If $a \in \lambda_0$, then $e^{-1}([0,a))$ is just equal to ${h_0}^{-1}([0,a))$.

\item If $a = y_{\delta}$ for some $\delta \in (\lambda+1)'$. Then $e^{-1}([0,a)) = k^{-1}([0,y_{\delta}))$, which is open by continuity of $k$. \end{enumerate}

We may prove $e^{-1}((b,e(\hat{x})])$ is open in a nearly identical way.

As $e: X \rightarrow Y$ is a continuous surjection, $Y$ must be compact. Thus by Theorem~\ref{D}, $Y$ is homeomorphic to an ordinal. This implies $X$ must be homeomorphic an ordinal as well. \end{proof}

We now know that every infinite, compact, simple, $X$ cleavable over an ordinal such that $X$ hereditarily has a spine must be homeomorphic to an ordinal. We will use this to show that every infinite compactum cleavable over an ordinal is homeomorphic to an ordinal in the following way: we will prove that every infinite compactum cleavable over an ordinal must hereditarily have a spine. 

\section{Finding the semi-spine of $X$}

In this section, we show that every infinite compactum cleavable over an ordinal must have a semi-spine. We do so by first finding a set $A \subseteq X$, indexing $A$ using ordinals, then using this index to find $T \subset A$; this set $T$, along with a function $f$ that we will define, will be our semi-spine. The main results of this section are contained in Theorem~\ref{sp23}. We will then show in the next section that $(T,f)$ is a spine. 

We begin with two definitions, an observation, and a well-known lemma. We then use these to construct the set we will encode.

\begin{definition} The \textbf{cofinality} of a partially ordered set $A$, $\cf(A)$, is defined as the least of the cardinalities of the cofinal subsets of $A$. \end{definition}

\begin{obs} \label{pcf} If a space $X$ is cleavable over an ordinal $\lambda$, then for every $x \in X$, if $f \neq g$ both cleave along $\left\{x\right\}$ over $\lambda$, then $\cf(f(x)) = \cf(g(x))$. \end{obs}

\begin{proof} Assume for a contradiction that $\cf(f(x)) < \cf(g(x))$. Let $\Delta \subset [0,f(X))$ be a cofinal subset of cardinality $\cf(f(x))$. For every $\delta \in \Delta$, let $x_{\delta} \in f^{-1}(\delta)$. Let $D = \left\{x_{\delta}: \delta \in \Delta \right\}$. Then $|D| = \cf(f(x))$, and as $f$ cleaves along $\left\{x\right\}$, $x \in \overline{D}$. Since $g$ cleaves along $\left\{x\right\}$, $g(D)$ must be cofinal in $g(x)$, implying $\cf(g(x)) \leq \cf(f(x))$, a contradiction. \end{proof}

\begin{definition} Let a space $X$ cleave along an ordinal $\lambda$, let $x \in X$, and let $f$ cleave along $\left\{x\right\}$ over $\lambda$. We say the \textbf{pseudo-cofinality} of $x$, $\pcf(x)$, is the least of the cardinalities of the cofinal subsets of $f(x)$. That is, $\pcf(x) = \cf(f(x))$. \end{definition}

Observation~\ref{pcf} implies this definition is well defined.

\begin{lemma} \label{singular} Every singular ordinal $\beta$ such that $\cf(\beta) = \gamma$ contains a cofinal club homeomorphic to $\gamma$. \end{lemma}

We now construct $A \subseteq X$, which we will use to find the spine of $X$.

\begin{quote} \textbf{Construction of $A \subseteq X$}:

Let $X$ be infinite, compact, simple with $\hat{x}$ such that $\pcf(\hat{x}) > \omega$, and such that $X$ is cleavable over an ordinal. Let $f$ cleave along $I_0(X) \cup \left\{\hat{x}\right\}$. Since $f$ is continuous and $X$ is compact, $|f^{-1}(f(y))|$ must be finite for all $y \in I_0(X) \cup \left\{\hat{x} \right\}$. Without loss of generality then, assume $f^{-1}(f(y)) = y$ for every $y \in I_0(X) \cup \left\{\hat{x}\right\}$. 

As $f(X)$ is a closed subset of an ordinal, by Lemma~\ref{ord} we know it is homeomorphic to an ordinal. Therefore, without loss of generality, let $f(X) = \mu+1 \subseteq \lambda$. If $\mu$ is singular, let $M \subset \mu$ be a cofinal club homeomorphic to $\cf(\mu)$. This exists by Lemma~\ref{singular}. Note $\mu > \omega$, as $\pcf(\hat{x}) > \omega$; if $\mu$ is regular, let $M = \mu$. 

Let $A$ be chosen in the following way: for each $\alpha \in M$, let $x_{\alpha}$ be such that $f(x_{\alpha}) = \alpha$. Let $A = \left\{x_{\alpha}: \alpha \in M \right\}$. \end{quote} 

We will use the subscripts of the elements of $A$, and the continuity of various functions from $X$ to $\lambda$, to show there exists a set $T \subseteq A$ such that $T \cup \left\{\hat{x}\right\}$ is closed in $X$. By construction, $f|_A$ is injective, therefore $f|_T$ will be injective; this will be all we need to give us that $(T,f)$ is a semi-spine of $X$. To see how we will use the subscripts of $A$, however, we must state a well-known lemma, also known as Fodor's Lemma.

\begin{lemma} [Pressing Down Lemma] Let $\kappa > \omega$ be regular, $S$ a stationary subset of $\kappa$, and $f: S \rightarrow \kappa$ such that $\forall \gamma \in S$, $f(\gamma) < \gamma$; then for some $\alpha < \kappa$, $f^{-1}(\alpha)$ is stationary. \end{lemma}

\begin{lemma} \label{pdl} Let $X$, $A$, $f$, and $M$ be as described in the \textbf{Construction of $A$}. If $j$ cleaves along $A$, then there exists a club subset of $M$ such that for every element $\beta$ in this club, and for every $\eta \in M$, $j(x_{\eta}) \geq j(x_{\beta})$ implies $\eta > \beta$. \end{lemma}

\begin{proof} If $j = f$, the function used in our construction of $A$, then we immediately know that $(A,f)$ is a spine of $X$. Therefore, assume $j \neq f$. Also assume without loss of generality that $j(\hat{x})$ is the greatest element of $j(X)$. We define a regressive function based on the subscripts of the elements of $A$ to show a club subset exists as described above. Remember that we have defined $M$ to be $\mu = f(X) \setminus \left\{f(\hat{x})\right\}$ if $\mu$ is regular, or a cofinal subset of $\mu$ homeomorphic to $\cf(\mu)$ if $\mu$ is singular. As we will soon see, it does not affect the proof whether we assume $\mu$ is singular or regular. Therefore, for ease of notation, we will assume $\mu$ is regular.

We construct a function from $\mu$ to $\mu$ which is regressive on a (possibly finite) subset. Let $g: \mu \rightarrow \mu$ be defined to be $$g(\alpha) = \min{\left\{\eta \in \mu: x_{\eta} \in j^{-1}([j(x_{\alpha}),j(\hat{x})] \right\}}$$.

That is, $g$ is the identity for those elements $\delta$ such that if $j(x_{\delta}) \leq j(x_{\gamma})$, then $\delta \leq \gamma$. Let $\hat{B} = \left\{\delta \in \mu: g(\delta) = \delta \right\}$. The function $g$ is clearly regressive on those ordinals in $\mu \setminus \hat{B}$. If $\mu \setminus \hat{B}$ were stationary, this would make $g$ regressive on a stationary subset of a regular ordinal. Therefore by the Pressing Down Lemma, an unbounded number of elements would be mapped to points less than some $x_{\gamma}$; this contradicts either the continuity of $j$, or the fact that $j$ cleaves along $A$, or that $j(\hat{x})$ is the greatest element of $j(X)$. Therefore $\mu \setminus B$ cannot be stationary, and $\hat{B}$ must contain a club. \end{proof}

We now use this lemma to construct $T \subseteq A$ such that $T \cup \left\{\hat{x}\right\}$ is closed in $X$. 

\begin{quote} \textbf{Construction of $T \subseteq A$}: 

Let $X$, $A$, $f$, $M$, and $\mu+1$ be as described in \textbf{Construction of $A$}, and let $j$ cleave along $A$. Let $C$ be the club subset of $\mu$ described in Lemma~\ref{pdl}. Note that by construction, $C$ is homeomorphic to $\cf(\mu)$. Furthermore, remember that the elements of $A$ are written as $x_{\alpha}$, where $f(x_{\alpha}) = \alpha \in M \subseteq \mu \subseteq \lambda$. 

Let $T = \left\{x_{\beta_{\gamma}}: \gamma \in M\right\}$, where $\beta_{\gamma}$ is defined as:

\textbf{Base Step}: Let $\beta_0$ be the least element of $C$ such that $\beta_0 > \eta$ for every $x_{\eta} \in j^{-1}([0,j(x_0)])$.

\textbf{Successor Step}: Let $\beta_{\alpha+1}$ be the least element of $C$ such that $\beta_{\alpha+1} > \eta$ for every $x_{\eta} \in j^{-1}([0,j(x_{\beta_{\alpha}})])$.

\textbf{Limit Step}: Consider step $\delta$, where $\delta$ is a limit ordinal. Since the ordinals we have chosen are increasing at every step, and $C$ is club in $\mu$, by construction there is only one element in $\overline{\left\{\beta_{\alpha}: \alpha < \delta \right\}} \setminus \left\{\beta_{\alpha}: \alpha < \delta \right\}$; let $\beta_{\delta}$ be this element. \end{quote}

We now explore the properties of $T$. 

\begin{lemma} \label{closed1} The set $\left\{\beta_{\alpha}: x_{\beta_{\alpha}} \in \hat{T} \right\}$ is club in $\mu$. \end{lemma}

\begin{proof} This is obvious by construction. \end{proof}

\begin{lemma} \label{less} For every $\gamma \in M$, $\delta \in M'$, and $\beta_{\delta} \in C$, $\gamma > \beta_{\delta}$ if and only if $j(x_{\gamma}) \geq j(x_{\beta_{\delta}})$. \end{lemma}

\begin{proof} This follows from the description of $C$ in Lemma~\ref{pdl}, and the construction of $T$. \end{proof}

We now rely on Lemma~\ref{less} to show that $T \cup \left\{\hat{x}\right\}$ is closed. It will then be an obvious consequence that $(T,f)$ is a semi-spine of $X$. 

\begin{lemma} \label{closed2} If $X$, $f$, $\hat{x}$, and $T$ are all as described in \textbf{Construction of $T$}, then $T \cup \left\{\hat{x}\right\}$ is closed in $X$. \end{lemma}

\begin{proof} Let $z \in \overline{T}$; we know if $z = \hat{x}$, then $z$ is already an element of our set. Therefore assume $z \neq \hat{x}$. Take $\left\{\beta_{\alpha}: j(x_{\beta_{\alpha}}) < z \right\}$. We know this set has to be unbounded in $j(z)$ by continuity of $j$. Let $\delta$ be the least ordinal greater than every $\alpha$ in this set, and consider $x_{\beta_{\delta}}$. 

By continuity of $f$ and $j$, and by construction of $T$, we know $\left\{\beta_{\alpha}: j(x_{\beta_{\alpha}}) < z \right\}$ must also be unbounded in $z$; therefore $j(z) = j(x_{\beta_{\delta}})$. (This is a result of how we chose $x_{\beta_{\delta}}$ in our construction of $A$, and our choice of $\beta_{\delta}$ in our construction of $T$.) Thus $z \in A$, and may be written as $x_{\eta}$. However, by Lemma~\ref{less}, $\eta$ must be greater than or equal to $\beta_{\delta}$. Since $z = x_{\eta}$ is in the closure of $\left\{\beta_{\alpha}: j(x_{\beta_{\alpha}}) < z \right\}$, $\eta$ cannot be greater than $\beta_{\delta}$; thus $\eta$ must equal $\beta_{\delta}$, and hence $z \in T$. $T$, therefore, contains all of its limit points. \end{proof}

Using these lemmas, we may now show that $(T,f)$ is a semi-spine of $X$. 

\begin{lemma} \label{sp23} $(T,f)$ is a semi-spine of $X$. \end{lemma}

\begin{proof} By Lemma~\ref{closed2}, $T \cup \left\{\hat{x}\right\}$ is closed in $X$; it is also obvious that $f|_{T \cup \left\{\hat{x}\right\}}$ is injective. Therefore $f|_{T}$ is an embedding. It is also obviously unbounded in $f(X) \setminus \left\{f(\hat{x})\right\}$. \end{proof}

The problem we are now faced with is if $(T,f)$ satisfies property $(1)$ of Definition~\ref{spine}. We consider this in the next section.

\section{Finding the Spine of $X$}

We now know that $(T,f)$ is a semi-spine of $X$. In order to show $(1)$ holds, however, we must ensure no elements of $X \setminus T$ are mapped onto the image of $T$. We do so by showing that if $X$ is in fact cleavable over $\lambda$, then for some closed subset of $f(X \setminus \left\{\hat{x}\right\})$, if $\beta$ is in this closed subset, then $|f^{-1}(\beta)| = 1$. This result is found in Lemma~\ref{sp1}. 

We begin, however, with a definition, and a few observations. 

\begin{definition} If $\beta$ is any ordinal, then we use $\beta^*$ to refer to $\beta$ with reversed order, and the order topology. For example, the least element of $(\omega+1)^*$ is $\omega$, and the greatest element is $0$. $\omega^*$ has no least element. \end{definition}

\begin{lemma} \label{left} If $X$ is an infinite compactum cleavable over an ordinal such that $X$ is simple with $\hat{x}$, then $X$ contains a subset homeomorphic to $\pcf(\hat{x}) + 1$. \end{lemma}

\begin{proof} This follows from Lemma~\ref{sp23}. \end{proof}

\begin{lemma} \label{left2} There do not exist ordinals $\alpha, \beta, \gamma$, where $\cf(\alpha)$ or $\cf(\beta)$ are uncountable, such that $X = \alpha + 1 + \beta^*$ is cleavable over $\gamma$. \end{lemma}

\begin{proof} Let $z = \left\{\alpha \right\} = \left\{\beta^*\right\}$; that is, $z$ is the point that joins $\alpha$ and $\beta^*$. If $f$ cleaves along $\left\{z\right\}$, we can immediately see that $\cf(\alpha)$ must equal $\cf(\beta)$. 

Now let $A = \alpha' \cup I_0(1+\beta^*)$. We claim no function can cleave along $A$.

To see this, assume for a contradiction that some function $g$ does cleave along $A$. Note that $g((\alpha+1)')$ and $g((1+ \beta^*)')$ must be closed in $g(X)$, since $X$ is compact, $\gamma$ is Hausdorff, and $g$ is continuous. Also note that from the way we have chosen $A$, $g(\alpha')$ and $g((\beta^*)')$ must be unbounded in $g(z)$: $g(\alpha')$ must be unbounded since $z \notin A$, and $g((\beta^*)')$ must be unbounded since $g$ is continuous, and $\cf(\beta)$ is uncountable. As $g(\alpha')$ and $g((\beta^*)')$ must be club in $g(z)$, $g(\alpha') \cap g((\beta^*)')$ must be non-empty, a contradiction since $g$ is assumed to cleave them apart. \end{proof}

\begin{lemma} \label{not2} If $X$ is an infinite compactum cleavable over an ordinal such that $X$ is simple with $\hat{x}$, and $\pcf(\hat{x}) > \omega$, then $X$ cannot contain two subsets homeomorphic to $\pcf(\hat{x}+1)$, $A$ and $B$, such that $A \cap B = \left\{\hat{x}\right\}$. \end{lemma}

\begin{proof} This follows from Lemmas~\ref{left} and \ref{left2}. \end{proof}

We rely on this lemma to prove $X$ must have a spine. We do so by trying to construct sets similar to $A$ and $T$, disjoint from $A$ and $T$, and showing that it cannot be done; we must therefore have that some subset of $T$, along with $f$, must function as a spine for $X$.

We begin by constructing those sets similar to $A$ and $T$.

\begin{quote} \textbf{Construction of $\hat{A}$ and $S \subset \hat{A}$}: 

Let $X$, $f$, and $M$ be as described in \textbf{Construction of $A$}. Let $\hat{A}$ be chosen in the following way: for each $\alpha \in M$, let $z_{\alpha}$ be an element of $f^{-1}(\alpha) \setminus A$ if $f^{-1}(\alpha) \setminus A$ is non-empty, and equal to $x_{\alpha}$ if $f^{-1}(\alpha) \setminus A$ is empty. 

Let $\hat{A} = \left\{z_{\alpha}: \alpha \in M \right\}$. 

Let $S \subseteq \hat{A}$ be constructed in the same way $T \subseteq A$ was constructed. \end{quote}

We will eventually show $(S \cap T,f)$ is a spine of $X$, and we begin to do so with a few consequences of Lemma~\ref{not2}.

\begin{lemma} \label{nonempty} The set $S \cap T$ is non-empty. \end{lemma}

\begin{proof} If $S \cap T$ were empty, then following from Lemma~\ref{left}, $S \cup \left\{\hat{x}\right\}$ and $T \cup \left\{\hat{x}\right\}$ would be homeomorphic to $\pcf(\hat{x}) + 1$. This implies $S \cup \left\{\hat{x}\right\} \cup T$ is cleavable over $\lambda$, but this contradicts Lemma~\ref{not2}. \end{proof}

We can do even better than showing $S \cap T$ is non-empty:

\begin{lemma} \label{club} $f(S \cap T)$ is club in $f(X) \setminus \left\{f(\hat{x}) \right\}$. \end{lemma}

\begin{proof} Unboundedness of $f(S \cap T)$ is proved in the exact same way as in Lemma~\ref{nonempty}. The fact that $f(S \cap T)$ is closed in $f(X) \setminus \left\{f(\hat{x}) \right\}$ follows from Lemma~\ref{closed2}, as $(S \cup \left\{\hat{x}\right\}) \cap (T \cup \left\{\hat{x}\right\})$ must be closed, and $f$ is continuous. \end{proof}

\begin{lemma} \label{sp1} For every $\beta \in f(S \cap T)$, $|f^{-1}(f(S \cap T))| = 1$. \end{lemma}

\begin{proof} This follows from how we have constructed $S$ and $T$. \end{proof}

\begin{lemma} \label{lemmaspine} $(S \cap T, f)$ is a spine for $X$. \end{lemma}

\begin{proof} By construction, $f|_{S \cap T}$ is injective. By Lemmas~\ref{closed2} and \ref{sp1}, $(S \cap T,f)$ satisfies all properties of Definition~\ref{spine}. \end{proof}

We are now ready to prove the main theorems of this paper, which we do in the next section.

\section{Results} 

In this section, we answer Questions~\ref{qone} and \ref{qtwo}, and characterize those infinite compacta that are cleavable over an ordinal.

\begin{thm} \label{thmspine} Every infinite compactum $X$ cleavable over an ordinal such that $X$ is simple with $\hat{x}$ has a spine. \end{thm}

\begin{proof} If $\pcf(\hat{x}) > \omega$, then following from Lemma~\ref{lemmaspine}, we know we may find a spine for $X$. Now assume $\pcf(\hat{x}) \leq \omega$. If $\pcf(\hat{x}) < \omega$, then $X$ must be finite, and is thus homeomorphic to a finite ordinal. If $\pcf(\hat{x}) = \omega$, then let $f$ cleave along $I_0(X)$. Since $f$ cleaves along all of the isolated points of $X$, we may assume without loss of generality by compactness that for every elements $x \in I_0(X)$, $f^{-1}(f(x)) = x$. (Since $|f^{-1}(f(x))|$ must at least be finite, we would otherwise be able to modify $f$ such that this assumption holds.) As $\pcf(\hat{x}) = \omega$, we may find a countable sequence $\left\langle a_n \right\rangle$ of isolated points that converges solely to $\hat{x}$. The pair $(\left\{a_n: n \in \omega \right\}, f)$ obviously satisfies all properties of Definition~\ref{spine}, and is therefore a spine for $X$. \end{proof}

\begin{thm} \label{hthmspine} Every infinite simple compactum $X$ cleavable over an ordinal hereditarily has a spine. \end{thm}

\begin{proof} Every closed, infinite, simple $A \subseteq X$ is also cleavable over an ordinal. \end{proof}

\begin{thm} \label{aamain} If $X$ is an infinite, simple, compactum cleavable over an ordinal, then $X$ is homeomorphic to an ordinal. \end{thm}

\begin{proof} By Lemma~\ref{hthmspine}, $X$ must hereditarily have a spine. By Lemma~\ref{amain}, $X$ must therefore be homeomorphic to an ordinal. \end{proof}

Me may now give a definitive answer to Question~\ref{qtwo} for cleavability over ordinals:

\begin{thm}\label{othm} If $X$ is an infinite compactum cleavable over an ordinal, then $X$ is homeomorphic to an ordinal, and is therefore a LOTS. \end{thm}

\begin{proof} The last non-empty derived set of $X$, call it $X^{\beta}$, must have finitely many elements. We may partition $X$ into finitely many clopen sets, such that each clopen set contains one element of $X^{\beta}$. Each clopen set, by Theorem~\ref{aamain}, must be homeomorphic to an ordinal; in fact, following from Theorem $2$ in \cite{D}, they must all be homeomorphic to the same ordinal. Therefore if $\lambda$ is the ordinal to which they are all homeomorphic, $X$ will be homeomorphic to $\lambda \cdot n + 1$, where $n = |X^{\beta}|$. \end{proof}

We may also improve the specificity of these results, and in Corollary~\ref{limit} we do so. However, before we state and prove Corollary~\ref{limit}, we must provide some new lemmas and definitions. Lemma~\ref{bumpuprank}, Definition~\ref{rdef}, and Theorem~\ref{CBrank} are the work of Richard Lupton, and the author is grateful for his help.

\begin{lemma} \label{bumpuprank} Suppose $X$ is scattered and $A \subseteq X$ with $x \in \overline{A} \setminus A$. Then $\textrm{CB}^*(x)$ must be greater than $\inf_{a \in A} ( \CB^*(a))$. \end{lemma}

\begin{proof} Let $\alpha = \inf_{a \in A} \CB^*(a)$. Then $A \subseteq X^\alpha$, and, since $X^\alpha$ is closed, $\overline{A} \subseteq X^\alpha$. In particular, $x \in X^\alpha$. We know $x$ is not isolated in $X^\alpha$, since $x$ is not an element of $A$, but every open set about $x$ must have non-empty intersection with $A$. So $x \in X^{\alpha + 1}$, and hence $\CB^*(x) \geq \alpha+1$. \end{proof}

\begin{definition} \label{rdef} We say that an ordinal $\alpha$ is an \textbf{even ordinal} if $\alpha$ is of the form $\lambda + (2 \cdot n)$, where $\lambda$ is a (necessarily unique) limit ordinal and $n$ is a natural number. Let us denote the class of even ordinals by $\mathbb{EON}$. \end{definition}

\begin{thm}\label{CBrank} Suppose $X$ and $Y$ are scattered and $X$ cleaves over $Y$. Then $\CB(X) \leq \CB(Y)$. \end{thm}

\begin{proof} Let $f : X \rightarrow Y$ cleave along $A$, where $A = \bigcup_{\alpha \in \mathbb{EON}} (X^\alpha \setminus X^{\alpha+1}).$ We show by transfinite induction on $\CB^*(x)$, that for each $x \in X$, $\CB^*(x) \leq \CB^*(f(x))$. It follows that $\CB(X) \leq \CB(Y)$.

The base case is clear since $0 \leq \CB^*(f(x))$ for all $x \in X$. Let us now suppose, as an inductive hypothesis, that for $x$ with $\CB^*(x) \leq \alpha$ we have $\CB^*(x) \leq \CB^*(f(x))$. Suppose $x$ satisfies $\CB^*(x) = \alpha+1$. Clearly $\left ( \alpha + 1 \in \mathbb{EON} \leftrightarrow \alpha \not\in \mathbb{EON} \right )$, so $f(x) \not\in f \left ( X^{\alpha} \setminus X^{\alpha+1} \right )$. However, $x \in \overline{X^{\alpha} \setminus X^{\alpha+1}}$, so by continuity of $f$, $f(x) \in \overline{f \left ( X^{\alpha} \setminus X^{\alpha+1} \right )}$. Therefore, by Lemma~\ref{bumpuprank}:

$$\CB^*(f(x)) > \inf_{y \in (X^{\alpha} \setminus X^{\alpha+1})} \CB^*(f(y)) \geq \alpha $$

\noindent where the last inequality is from our inductive hypothesis. Hence $\CB^*(f(x)) \geq \alpha + 1 = \CB^*(x)$.

Finally, suppose $\lambda$ is a limit ordinal and, as an inductive hypothesis, for any $x \in X$ with $\CB^*(x) < \lambda$ we have $\CB^*(x) \leq \CB^*(f(x))$. Observe that $\lambda \in \mathbb{EON}$. Suppose $x \in X^\lambda \setminus X^{\lambda + 1}$, so $\CB^*(x) = \lambda$. Suppose $\alpha < \lambda$. As $\lambda$ is a limit ordinal, there is an ordinal $\beta$ with $\beta \not\in \mathbb{EON}$, and $\alpha \leq \beta < \lambda$ (one of $\alpha$ or $\alpha + 1$ will work). In particular, $f(x) \not\in f \left ( X^\beta \setminus X^{\beta+1} \right )$. Nonetheless, $x \in \overline{X^\beta \setminus X^{\beta+1}}$, hence by continuity of $f$, $f(x) \in \overline{f \left ( X^\beta \setminus X^{\beta+1} \right )}$. By Lemma~\ref{bumpuprank} and inductive hypothesis, 
$$\CB^*(f(x)) > \textnormal{inf}_{y \in \left ( X^\beta \setminus X^{\beta+1} \right )} \CB^*(f(y)) \geq \beta \geq \alpha.$$
Since $\alpha < \lambda$ was arbitrary, $\CB^*(f(x)) \geq \lambda$, concluding the induction. \end{proof}

From Theorem~\ref{CBrank}, we have the following:

\begin{cor} \label{limit} If $\mu$ is the least ordinal over which an infinite compactum $X$ is cleavable, then $X$ must be homeomorphic to an ordinal less than $\mu \cdot \omega$. \end{cor}

\section{Conclusions and Open Questions}

We have now shown that if $X$ is an infinite compactum cleavable over an ordinal $\lambda$, then $X$ is homeomorphic to an ordinal and is therefore a LOTS. If either $X$ or $\lambda$ is countable, then from the results within [Levine, \textit{Cleavability and scattered first-countable LOTS}, unpublished], we know the specific conditions under which $X$ is embeddable into $\lambda$. The following, however, is still an open question:

\begin{question}If $X$ is an infinite compactum cleavable over an uncountable ordinal $\lambda > \omega_1$, must we have that $X$ is embeddable into $\lambda$? \end{question}

While we do not have a definitive answer for this question, we do know the following:

\begin{lemma} If $\beta$ is an uncountable regular ordinal, with $\cf(\beta) > \omega$, then $\beta \cdot j + 1$, where $j \in \left\{2,3,4 \right\}$ is not cleavable over $\beta \cdot k + 1$, where $k < j$. \end{lemma}

\begin{proof} We will prove the lemma true in the case where $j = 4$ and $k = 3$. The proof may then be modified in the case where $j$ is $2$ or $3$.

Let $C$ be club in $\beta$; similarly, let $C_{m+1}$ be a copy of $C$ in $(\beta \cdot m, \beta \cdot (m+1))$, where $m+1 \in \left\{1,2,3,4\right\}$. Let $A = C_1 \cup C_3 \cup ((\beta, \beta \cdot 2)\setminus C_2) \cup ((\beta \cdot 3,\beta \cdot4) \setminus C_4) \cup \left\{\beta, \beta \cdot 2 \right\}$. That is, $A$ contains the clubs $C_1$ and $C_3$, the compliments of $C_2$ and $C_4$ in their respective intervals, and the endpoints $\beta$ and $\beta \cdot 2$. I claim no function can cleave $X$ along $A$ over $\beta \cdot 3 + 1$.

Assume for a contradiction that there does exist such an $f$. Then from the way we have chosen $A$, $f(C_{m+1})$ must be unbounded in $f(\beta \cdot (m+1))$ for every $m+1 \in \left\{1,2,3,4\right\}$. As $\beta$ is regular, this implies $f(\beta \cdot (m+1))$ must be mapped to one of the points $\beta \cdot m$, and as $A$ and $X \setminus A$ both contain two of the $\beta \cdot (m+1)$, we know $f(\beta) = f(\beta \cdot 2)$ and $f(\beta \cdot 3) = f(\beta \cdot 4)$. Furthermore, as $f$ is continuous, $X$ is compact, each $C_{m+1}$ is closed, and $\beta \cdot 3 +1$ is $T_2$, $f(C_{m+1})$ must be closed. Thus, by continuity of $f$, $f(C_1)$ and $f(C_2)$ must both be club in $f(\beta)$, and therefore $f(C_1) \cap f(C_2)$ must be non-empty. However, $A$ contains $C_1$ and $((\beta, \beta \cdot 2)\setminus C_2)$, contradicting the fact that $f$ cleaves along $A$. Thus no $f$ can cleave along $A$. \end{proof}

It seems likely, however, that the following is true:

\begin{question} If $\beta$ is an uncountable ordinal, must it be the case that $\beta \cdot n + 1$, where $n \geq 5$, is cleavable over $\beta \cdot m +1$, where $m \geq 4$, but is not embeddable unless $n \leq m$? \end{question}
 
\bibliographystyle{elsarticle-num}
\nocite{*}
\bibliography{ordinals}

\end{document}